\documentclass[12pt, twoside]{article}
\usepackage{amsmath,amsthm}
\usepackage{amssymb,latexsym}
\usepackage{enumerate}
\usepackage{indentfirst}

\newcommand{\mes}{\operatorname{mes}}

\newtheorem{theorem}{Theorem}
\newtheorem{corollary}{Corollary}
\newtheorem{lemma}{Lemma}

\begin{document}

\pagestyle{myheadings}
\markboth{D. Kaliada, F. G\"otze and O. Kukso}
{Number of integer cubic polynomials with bounded discriminants}

\title{The asymptotic number of integral cubic polynomials with bounded heights and discriminants}

\author{
Dzianis Kaliada
\and 
Friedrich G\"otze
\and
Olga Kukso
}

\maketitle

\renewcommand{\thefootnote}{}



\renewcommand{\thefootnote}{\arabic{footnote}}
\setcounter{footnote}{0}

\begin{abstract}
Let $P$ denote a cubic integral polynomial, and let $D(P)$ and  $H(P)$ denote the discriminant and height of $P$ respectively.
Let $N(Q,X)$ be the number of cubic integer polynomials $P$ such that $H(P)\le Q$ and $|D(P)|\le X$.
We obtain the asymptotic formula of $N(Q,X)$ for $Q^{14/5} \ll X \ll Q^4$ and as $Q\to \infty$.
Using this result, for $0\le \eta\le 0.9$ we prove that
\[
\sum_{\stackrel{\scriptstyle H(P)\le Q}{1\le |D(P)|\ll Q^{4-\eta}}}|D(P)|^{-1/2}\ \asymp Q^{2-\frac\eta3}
\]
for all sufficiently large $Q$, where the sum is taken over irreducible polynomials. This improves upon a result of Davenport who dealt with the case $\eta=0$.
We also consider an application of the main theorem to some outstanding problems of transcendental number theory.
\end{abstract}

\section{Introduction and results}
Let $P(x)=a_{n}x^{n}+\dots+a_{1}x+a_{0} \in \mathbb{Z}[x]$ denote a polynomial of degree $n$,
let $H(P) = \max_{0\le i\le n} |a_{i}|$ denote the height of $P$, and
let $\alpha_{1},\alpha_{2},\dots,\alpha_{n} \in \mathbb{C}$ denote the roots of $P$. The number
\begin{equation}\label{er1}
 D(P)=a_{n}^{2n-2}\prod_{1\le i<j\le n}(\alpha_{i}-\alpha_{j})^{2}
\end{equation}
is called the discriminant of $P$. Properties of $D(P)$ where $P$ is an integral
polynomial have numerous applications in transcendental number theory. In particular,
B.~Volkmann's proof \cite{Vol} of the cubic case of Mahler's conjecture \cite{Mahler-1932b}
was based purely on a summation formula for discriminants proved by H.~Davenport in \cite{Dav61}.
The behavior of $D(P)$ is also closely related to the problem of the separation of conjugate
algebraic numbers, which has been recently studied in some depth in \cite{Beresnevich-Bernik-Goetze-10,BM04,BM10,E04}.
Recently lower bounds for the number of integral polynomials with given heights and discriminants
(or a discriminant divisible by a large prime power) have been obtained in \cite{BGKacta, BGKlitva}.

Staying strictly within polynomials of degree 3, we obtain the asymptotic formula for the number of integral
polynomials with bounded discriminants (Theorem~\ref{thm1} below).
Using this asymptotics, we extend Davenport's summation formula (Theorem~\ref{thm2} below).
In the last section we consider an application of our main result to finding the Hausdorff dimension
of real numbers with a certain approximation property by cubic polynomials.

Throughout, $\# M$ denotes the number of elements in a set $M$, and $\mes_{k} M$ denotes the $k$--dimensional Lebesgue measure
of a set $M \subset \mathbb{R}^{n}$ ($k \le n$). We will also use the Vinogradov symbol $\ll$.
The expression $f \ll g$ is equivalent to that the inequality $f \le c g$ holds for some absolute constant $c$. The expression $f \asymp g$ indicates that $g \ll f \ll g$. The expressions like $f \ll_{x_1,\dots,x_k} g$ or $f \asymp_{x_1,\dots,x_k} g$ mean that corresponding implicit constants depend only on parameters $x_1,\dots,x_k$.

Given $n \in \mathbb{N}$, $Q>1$ and $v\ge0$, define
\begin{gather}
\mathcal{P}_{n}(Q)=\{P \in \mathbb{Z}[x] : \deg P = n, \ H(P) \le Q\},\\
\mathcal{P}_{n}(Q,v)=\{P \in \mathcal{P}_{n}(Q) : |D(P)|\le \gamma_n Q^{2n-2-2v}\}, \label{er2}
\end{gather}
where the constant $\gamma_n$ depends only on the degree $n$ and is defined by
\[
\gamma_n := \sup_{\substack{
P \in \mathbb{Z}[x]\\
\deg P = n
}} \frac{|D(P)|}{(H(P))^{2n-2}}.
\]
Note, it is known (see \cite{Waer}) that the discriminant $D(P)$ is a homogeneous polynomial of degree $2n-2$ in coefficients of $P$. Therefore, $\gamma_n < +\infty$, and $H(P)\le Q$ obviously implies $|D(P)|\le \gamma_n Q^{2n-2}$.

The following lower bound for $\# \mathcal{P}_n(Q, v)$ has been shown in \cite{BGKacta} and \cite{BBG10}:
\begin{equation}\label{vb1}
\# \mathcal{P}_n(Q, v) \gg_n Q^{n+1-2v},
\end{equation}
where $0 < v < \frac{1}{2}$. Using the recent results of Beresnevich \cite{Beresnevich-SDA1} for the number
of rational points near non-degenerate analytic manifolds in $\mathbb{R}^n$, the validity of  \eqref{vb1}
can also be extended to the range $0<v<1$.
In \cite{K10} it was proved that $\# \mathcal{P}_3(Q,v) \ll Q^{4-5v/3}$.
Heuristic arguments suggested that the estimate for $\#\mathcal{P}_3(Q,v)$ in \eqref{vb1} is the best possible up to a constant, and the result from \cite{K10} doesn't contradict it.
However, the following main result of this paper, which gives upper and lower
bounds for $\# \mathcal{P}_3(Q, v)$, shows that this expectation is clearly wrong.

Let us define the following quantity
\[
N(X) = N(Q,X) := \# \{ P \in \mathcal{P}_3(Q) : |D(P)| \leq X \}.
\]
\begin{theorem}\label{thm1}
For $X$ satisfying $0 \le X \le Q^4 / 27$, the following equality holds:
\begin{equation}
N(Q,X) = \kappa\, Q^{2/3} \, X^{5/6} + O\left(X\left|\ln\left(Q^4 / X\right)\right| + X + Q^3\right),
\end{equation}
where $\kappa$ is an absolute constant defined by
\begin{equation}\label{eq_kappa}
\kappa = (3^{4/3} - 2)\cdot\frac{2^{7/3}}{\sqrt{3}}
\left(
2\int_{1}^{\infty} \frac{dt}{\sqrt{t^3+1}+\sqrt{t^3-1}} +
\int_{-1}^{1} \sqrt{t^3+1}\,dt
\right); 
\end{equation}
an implicit constant in the symbol $O(\cdot)$ is also absolute.
\end{theorem}

\begin{corollary}\label{cor1}
For any $v \in \left[0, \frac{3}{5}\right)$ and all sufficiently large $Q$, we have
\begin{equation}\label{er27}
 \# \mathcal{P}_3(Q, v) \asymp Q^{4-\frac{5}{3}v},
\end{equation}
where an implicit constant in the symbol $\asymp$ is absolute.
\end{corollary}

We shall use the result of Theorem \ref{thm1} to prove the following generalization of Davenport's summation formula \cite{Dav61} for $|D(P)|^{-1/2}$.

\begin{theorem}\label{thm2}
For $X$ satisfying $0 \le X \le Q^4 / 27$ and all sufficiently large $Q$, we have
\begin{equation}\label{er27+}
\sum_{\stackrel{\scriptstyle H(P)\le Q}{1\le |D(P)|\le X}}|D(P)|^{-1/2} =
 \frac{7}{6}\kappa\, Q^{2/3} X^{1/3} + O\left(X^{1/2}\left|\ln\left(\frac{Q^4}{X}\right)\right| + X^{1/2} + Q^{1.7}\right),
\end{equation}
where the summation is taken over irreducible polynomials $P$ of degree $3$; the constant $\kappa$ is defined by \eqref{eq_kappa}; an implicit constant in the symbol $O(\cdot)$ is absolute.
\end{theorem}

Assuming $X = \gamma_3 Q^{4-\eta}$ in \eqref{er27+} we have Davenport's formula \cite{Dav61} corresponds to the case $\eta=0$. For $0 < \eta < 0.9$ we obtain the following result.
\begin{corollary}\label{cor2-1}
Suppose that $0\le \eta < 9/10$. Then for all sufficiently large $Q$
\begin{equation}\label{er27+2}
 \sum_{\stackrel{\scriptstyle H(P)\le Q}{1\le |D(P)|\le \gamma_3 Q^{4-\eta}}}|D(P)|^{-1/2}\ \asymp Q^{2-\frac\eta3},
\end{equation}
where an implicit constant in the symbol $\asymp$ is absolute.
\end{corollary}

\section{Auxiliary statements}
For the proof of Theorem \ref{thm1} we shall need the following lemmas. The expression $\|\mathbf{x}\|_{\infty}$ denotes the maximum norm of a vector $\mathbf{x}=(x_1,\ldots,x_n)\in\mathbb{R}^n$.
\begin{lemma}\label{lm_dmn_mes}
Let $f({\bf x})\in C(\mathbb{R}^{n})$ and let $f(t {\bf x}) = t^d f({\bf x})$
for all $t \in \mathbb{R}$, $t > 0$.
Let
\[
G(\delta, R) = \{{\bf x} \in \mathbb{R}^{n} : |f({\bf x})| \le \delta, \ \|{\bf x}\|_{\infty} \le R \}
\]
and
\[
\sigma(\delta) = \mes_{n-1}\{{\bf x} \in \mathbb{R}^{n} : |f({\bf x})| \le \delta, \ \|{\bf x}\|_{\infty} = 1 \}.
\]
Then
\begin{equation}\label{er3}
\mes_{n} G(\delta,R)=\int_{0}^{R}r^{n-1}\sigma\left(\frac{\delta}{r^{d}}\right)dr.
\end{equation}
\end{lemma}

\begin{proof}
Let us consider the subsets of $G(\delta, R)$:
\[
G_i(\delta,R) = \left\{\mathbf{x}=(x_1,\ldots,x_n)\in G(\delta,R) : |x_i|\ge |x_k|, i\ne k\right\}.
\]
Obviously, we have
\[
\mes_{n} G(\delta,R) = \sum_{i=1}^n \mes_{n} G_i(\delta,R).
\]

Without loss of generality we consider the case of $G_1(\delta, R)$
\[
\mes_{n} G_1(\delta,R) = \int_{G_1(\delta,R)} dx_1\,dx_2\ldots dx_n.
\]
We change variables by formulas
$|x_1| = r$,
$x_i = r \theta_i$,  $2\le i\le n$.
Jacobian of this transformation is equal to
$\frac{\partial(x_1,x_2,\ldots,x_n)}{\partial(r,\theta_2,\ldots,\theta_n)} = r^{n-1}$.
Thus, we have
\[
\mes_{n} G_1(\delta,R) = \int_0^R \left( \int_{\Gamma_1^+(\delta,r)} d\theta_2\ldots d\theta_n  + 
\int_{\Gamma_1^-(\delta,r)} d\theta_2\ldots d\theta_n \right) r^{n-1} dr,
\]
where
\begin{align*}
\Gamma_1^+(\delta, r) &:= \{ (\theta_2,\ldots,\theta_n)\in\mathbb{R}^{n-1}:|\theta_i|\le 1, \ |f(r,r\theta_2,\ldots,r\theta_n)|\le\delta\},\\
\Gamma_1^-(\delta, r) &:= \{ (\theta_2,\ldots,\theta_n)\in\mathbb{R}^{n-1}:|\theta_i|\le 1, \ |f(-r,r\theta_2,\ldots,r\theta_n)|\le\delta\}.
\end{align*}

Let us denote
$\sigma_i(\delta) = \mes_{n-1}\{{\bf x} \in \mathbb{R}^{n} : |f({\bf x})| \le \delta, \ \|{\bf x}\|_{\infty} = 1, \ |x_i|=1 \}$. It is clear that $\sigma(\delta)=\sum_{i=1}^n \sigma_i(\delta)$. Since the function $f$ is homogeneous, we obtain
\[
\int_{\Gamma_1^+(\delta,r)} d\theta_2\ldots d\theta_n + \int_{\Gamma_1^-(\delta,r)} d\theta_2\ldots d\theta_n
= \sigma_1\!\left(\frac{\delta}{r^d}\right).
\]

The lemma is proved.
\end{proof}

Let us denote ${\bf p} = (p_n, \ldots, p_1, p_0) \in \mathbb{R}^{n+1}$.
Let us define a mapping $R : \mathbb{R}^{n+1} \to \mathbb{R}^{n+1}$, $R{\bf p} = (p_0, \ldots, p_{n-1}, p_n)$.
There is an equivalent definition of the mapping $R$:
\[
P(x) = \sum_{i=0}^{n} p_i x^i \longmapsto R(P)(x) = \sum_{i=0}^{n} p_{n-i} x^i = x^n P\left(\frac{1}{x}\right).
\]
\begin{lemma}\label{lm_dscr_prp}
Let $D({\bf p})$ denote a discriminant of a polynomial $P(x) = \sum_{i=0}^{n} p_i x^i$ as function of the
coefficients of the polynomial.
Then
\begin{equation}\label{eq-prty}
 D(-{\bf p}) = D({\bf p}),
\end{equation}
\begin{equation}\label{eqRprty}
 D(R{\bf p}) = D({\bf p}).
\end{equation}
\end{lemma}

\begin{proof}
The functional equation \eqref{eq-prty} directly follows from \eqref{er1}.

We will prove the equation \eqref{eqRprty}.
It is easy to see that the polynomial $Q(x) = R(P)(x)$ has roots $\beta_i = 1/\alpha_i$,
$1 \le i \le n$, and leading coefficient $q_n = p_0 = p_n \prod_{i=1}^n \alpha_i$. We put $q_n$ and $\beta_i$, $1 \le i \le n$,
into \eqref{er1} and obtain the equation \eqref{eqRprty}.
Note, in the proof we assume that $\alpha_i \ne 0$, $1 \le i \le n$, i.e. $p_0 \ne 0$. But discriminant is a continuous
function (polynomial) of the coefficients ${\bf p}$. Hence we have that the equation \eqref{eqRprty} is true in the case $p_0 = 0$.
\end{proof}

Given a polynomial
\begin{equation}\label{er5}
P(t) = z t^3 + y t^2 + x t + u
\end{equation}
of degree 3, or binary cubic form
\begin{equation}
P(\tau, \chi) = z \tau^3 + y \tau^2 \chi + x \tau \chi^2 + u \chi^3,
\end{equation}
its discriminant is well known to be
\begin{equation}\label{er6}
D(P)=x^{2}y^{2}-4uy^{3}-27u^{2}z^{2}-4x^{3}z+18uxyz.
\end{equation}

Let us consider the discriminant surface given by $D(P) = \delta $ in the space $\mathbb{R}^4$; that is
\begin{equation}\label{er7}
\mathcal{S}(\delta) = \left\{ (u,x,y,z)\in\mathbb{R}^4 : x^{2}y^{2}-4uy^{3}-27u^{2}z^{2}-4x^{3}z+18uxyz=\delta \right\}.
\end{equation}

Since the polynomial \eqref{er6} is quadratic  with respect to $u$, we may solve \eqref{er7} with respect to $u$.
\begin{lemma} The surface $\mathcal{S}(\delta)$ given by \eqref{er7} in the $\mathbb{R}^4$ has the explicit form
\begin{equation}\label{er8}
 \begin{matrix}
  u_{1}(x,y,z)=u_{1}(x,y,z,\delta)=\frac{9xyz-2y^3-\sqrt{S-27z^2\delta}}{27z^2},\\
  u_{2}(x,y,z)=u_{2}(x,y,z,\delta)=\frac{9xyz-2y^3+\sqrt{S-27z^2\delta}}{27z^2},
 \end{matrix}
\end{equation}
where $u_1$, $u_2$ are the two branches of the function $u$, and
\begin{equation}\label{er9}
 S=S(x,y,z):=4(y^{2}-3xz)^3.
\end{equation}
The domain of definition of $u_{1}(x,y,z)$ and $u_{2}(x,y,z)$ is given by
\begin{equation}\label{er10}
S-27z^{2}\delta \ge 0.
\end{equation}
\end{lemma}

The following two Lemmas \ref{lm_uz_less_1} and \ref{lm_uy_less_1} will be used to obtain the lower bound.
\begin{lemma}\label{lm_uz_less_1}
Let
\begin{equation}
M_{z}(\delta)=\left\{(x,y)\in\mathbb{R}^2 : |x| \le 1-\sqrt{\frac{|\delta|}{3}}, \ |y| \le 1, \ S(x,y,1)
\ge 27\delta \right\}.
\end{equation}
Then $|u_{1,2}(x,y,1,\delta)|\le 1$ for all $(x, y) \in M_{z}(\delta)$ for $|\delta| \le 1/27$.
\end{lemma}

\begin{proof}
It is easy to see that
\[
 |u_{1,2}(x,y,1,\delta)| \le 
 \frac{9|xy|+2|y|^{3}+\sqrt{S}+\sqrt{27|\delta|}}{27}.
\]
The conditions $|x|\le 1$, $|y|\le 1$ imply the inequality $0 \le S \le 4(y^{2}+3|x|)^{3} \le 16^{2}$.
Hence, we have $|u_{1,2}(x,y,1,\delta)| \le \frac{9|x|+18+\sqrt{27|\delta|}}{27}$, and therefore, the condition
$9|x|+\sqrt{27|\delta|} \le 9$ implies the desired bound $|u_{1,2}(x,y,1,\delta)|\le 1$.
This condition is equivalent to $|x| \le 1-\sqrt{\frac{|\delta|}{3}}$. The lemma is proved.
\end{proof}

\begin{lemma}\label{lm_uy_less_1}
Let
\begin{equation}
M_{y}(\delta)=\left\{(x,z)\in\mathbb{R}^2 : |x| \le 1, \ \frac{1}{3} \le |z| \le 1, \ \frac{2}{9} \le xz, \
S(x,1,z)\ge 27z^{2}\delta \right\}.
\end{equation}
Then $|u_{1,2}(x,1,z,\delta)| \le 1$ for all $(x, z) \in M_{y}(\delta)$ for $|\delta|\le 1/27$.
\end{lemma}

\begin{proof}
It is easy to observe that
\[
 |u_{1,2}(x,1,z,\delta)| \le 
 \frac{|9xz-2|+\sqrt{S}+\sqrt{27z^2|\delta|}}{27z^2},
\]
where $S = 4(1-3xz)^3$.

From $S - 27z^2 \delta \ge 0$, we get the inequality $3xz \le 1-3 \cdot \sqrt[3]{\frac{z^2}{4} \delta}$.
The inequality $xz \ge 2/9$ yields $S \le 4/27$.
\begin{multline*}
|u_{1,2}(x,1,z,\delta)| \le \frac{1}{27z^2} \left(1 + \frac{2}{3\sqrt{3}} -
9 \cdot \sqrt[3]{\frac{z^2 \delta}{4}} + \sqrt{27z^2|\delta|} \right)\le\\
\le \frac{1}{3} \left(1 + \frac{2}{3\sqrt{3}}\right) +
\sqrt[3]{\frac{3 |\delta|}{4}} + \sqrt{\frac{|\delta|}{3}}.
\end{multline*}
For $|\delta| \le 1/27$ for all $(x,z) \in M_{y}(\delta)$ it holds $|u_{1,2}(x,1,z,\delta)| \le 1$.
\end{proof}

For any given ${\bf p} = (u, x, y, z) \in \mathbb{R}^4$ we write
$D({\bf p}) = x^{2}y^{2}-4uy^{3}-27u^{2}z^{2}-4x^{3}z+18uxyz$ (cf. \eqref{er6}).

\begin{lemma}\label{sgm_up_bound}
Let $\sigma(\delta) = \mes_3\{{\bf p}\in\mathbb{R}^4 : |D({\bf p})| \le \delta, \ \|{\bf p}\|_{\infty} = 1\}$.
Then for $0<\delta\le 1/27$ we have
\begin{equation}\label{er13}
 \sigma(\delta) = c_{1}\delta^{5/6} + O(\delta),
\end{equation}
where $c_1$ is an absolute constant, which doesn't depend on $\delta$;
an implicit constant in the symbol $O(\cdot)$ is also absolute.
\end{lemma}

\begin{proof}
Let $\delta > 0$, and let ${\bf p}=(u,x,y,z)$.

The properties of discriminant (see Lemma \ref{lm_dscr_prp}) imply that we need to consider
only two faces $z=1$ and $y=1$ of the box $\|{\bf p}\|_{\infty} = 1$.
Let us define
\[
\begin{gathered}
\sigma_z(\delta) = \mes_3\{{\bf p}\in\mathbb{R}^4 : |D({\bf p})| \le \delta, \ \|{\bf p}\|_{\infty} = 1,\ z=1\},\\
\sigma_y(\delta) = \mes_3\{{\bf p}\in\mathbb{R}^4 : |D({\bf p})| \le \delta, \ \|{\bf p}\|_{\infty} = 1,\ y=1\}.
\end{gathered}
\]
Then
\[
  \sigma(\delta)= 4(\sigma_z(\delta)+ \sigma_y(\delta)).
\]

Obviously, we have
\[
\sigma_z(\delta) = \iiint_{\mathcal{D}_z} du\, dx\, dy, \qquad
\sigma_y(\delta) = \iiint_{\mathcal{D}_y} du\, dx\, dz,
\]
where
\[
\begin{gathered}
\mathcal{D}_z := \{(u,x,y)\in\mathbb{R}^3 : |D(u,x,y,1)| \le \delta, \ \max\{|u|,|x|,|y|\}\le 1\},\\
\mathcal{D}_y := \{(u,x,z)\in\mathbb{R}^3 : |D(u,x,1,z)| \le \delta, \ \max\{|u|,|x|,|z|\}\le 1\}.
\end{gathered}
\]

For $\delta > 0$ we consider the auxiliary function
\begin{equation*}
h(\delta)=h(x,y,z,\delta)=
\left\{ \begin{array}{ll}
   (u_{1}(\delta)-u_{1}(-\delta))+(u_{2}(-\delta)-u_{2}(\delta)), &
    S > 27 z^{2}\delta,\\
   u_{2}(-\delta)-u_{1}(-\delta), &
    |S| \le  27 z^{2}\delta,
  \end{array} \right.
\end{equation*}
where $u_j(\delta) = u_j(x,y,z,\delta)$, $j=1,2$.

In the notation introduced above, using Lemmas \ref{lm_uz_less_1} and \ref{lm_uy_less_1}, we have
\[
\begin{aligned}
\iint_{M_{z}}h(x,y,1,\delta)\,dx\, dy \le \sigma_z(\delta) \le \iint_{G_z} h(x,y,1,\delta)\,dx\,dy, \\
\iint_{M_{y}}h(x,1,z,\delta)\,dx\, dz \le \sigma_y(\delta) \le \iint_{G_y} h(x,1,z,\delta)\,dx\,dz,
\end{aligned}
\]
where
\[
\begin{aligned}
G_z=G_z(\delta):=&\{(x,y)\in\mathbb{R}^2 : \max\{|x|,|y|\}\le 1, \ S(x,y,1) \ge - 27 \delta\},\\
G_y=G_y(\delta):=&\{(x,z)\in\mathbb{R}^2 : \max\{|x|,|z|\}\le1, \ S(x,1,z) \ge -27 z^2 \delta\}.
\end{aligned}
\]
and
\begin{align*}
M_z := & M_z(-\delta) = \left\{(x,y)\in G_z(\delta) : |x| \le 1-\sqrt{\frac{\delta}{3}}\right \},\\
M_y := & M_y(-\delta) = \left\{(x,z)\in G_y(\delta) : \frac{1}{3} \le |z|, \ \frac{2}{9} \le xz \right\}.
\end{align*}

Now, we are ready to prove asymptotic formulas for $\sigma_z(\delta)$ and $\sigma_y(\delta)$.
Note that the function $h(x,y,z,\delta)$ has the form
\begin{equation}\label{eq_hxyzd}
 h(x,y,z,\delta)= \left\{ \begin{array}{ll}
                \dfrac{4\delta}{\sqrt{S+27 z^{2}\delta}+\sqrt{S-27 z^2\delta}}, & \text{if }S> -27 z^2\delta,\\[4ex]
        \dfrac{2\sqrt{S+27 z^{2}\delta}}{27 z^{2}}, & \text{if } |S| \le 27 z^2 \delta,
               \end{array} \right.
\end{equation}
where the function $S=S(x,y,z)$ is defined by \eqref{er9}.

In calculations of integrals we will use the substitution $\frac{S}{27 z^{2}\delta}=t^{3}$,
which leads to the following equalities:
\begin{equation}\label{er14}
\frac{1}{3}\sqrt[3]{\frac{4}{z^{2}\delta}}\, (y^{2}-3xz) = t; \hspace*{3ex}
x = \frac{y^{2}}{3z}-\sqrt[3]{\frac{\delta}{4z}}\,t; \hspace*{3ex}
dx = -\sqrt[3]{\frac{\delta}{4z}}\,dt.
\end{equation}

\begin{lemma}\label{lm_sgmaz}
For $0<\delta\le 1/27$, we have
\begin{equation}\label{eq_sgmaz}
\sigma_z(\delta) = c_z \cdot \delta^{5/6} + O(\delta),
\end{equation}
where
\begin{equation}\label{eq_cz}
c_z = \frac{4^{2/3}}{\sqrt{27}} \left(
2 \int_{1}^{\infty} \frac{dt}{\sqrt{t^3+1}+\sqrt{t^3-1}} +
\int_{-1}^{1} \sqrt{t^3+1}\,dt
\right).
\end{equation}
\end{lemma}

\begin{proof}
Firstly, we calculate the integral
$\iint_{G_z} h(x,y,1,\delta)\,dx\,dy$.
Accordingly to \eqref{eq_hxyzd}, we divide the domain $G_{z}$ into two subdomains $G_{z} = G_z^{(1)} \sqcup G_z^{(2)}$,
where
\begin{align*}
G_z^{(1)} &= \{(x,y) \in G_{z} : S(x,y,1) > 27\delta\}, \\
G_z^{(2)} &= \{(x,y) \in G_{z} : |S(x,y,1)| \le 27\delta\}.
\end{align*}

For the domain $G_z^{(1)}$ we have $I_z^{(1)} := \iint_{G_z^{(1)}}h(x,y,1,\delta)dxdy$.
We apply the substitution \eqref{er14}.
After this transformation we obtain
\begin{equation*}
 I_z^{(1)} = \delta^{\frac{5}{6}}\cdot \frac{2\cdot 4^{\frac{2}{3}}}{\sqrt{27}}\int_0^1 dy \int_{1}^{\tau(y)}\frac{dt}{\sqrt{t^{3}+1}+\sqrt{t^{3}-1}},
\end{equation*}
where $\tau(y) := \frac13 \sqrt[3]{\frac4\delta}(y^2+3)$.

Since $\tau(y)\ge \tau_0 := \sqrt[3]{\frac4\delta}$ and
$
\int_{\tau(y)}^{\infty}\frac{dt}{\sqrt{t^{3}+1}+\sqrt{t^{3}-1}} \le
\int_{\tau_0}^{\infty} t^{-3/2} dt = 2^{2/3} \delta^{1/6},
$
we have the following asymptotics for $I_z^{(1)}$:
\begin{equation}\label{er15}
 I_z^{(1)} = \delta^{\frac{5}{6}}\cdot \frac{2\cdot 4^{\frac{2}{3}}}{\sqrt{27}}\int_{1}^{\infty}\frac{dt}{\sqrt{t^{3}+1}+\sqrt{t^{3}-1}} + O(\delta).
\end{equation}

Applying \eqref{er14} yields
\begin{equation}\label{er16}
 I_z^{(2)} := \iint_{G_z^{(2)}}h(x,y,1,\delta)dxdy=\delta^{\frac{5}{6}} \cdot \frac{4^{\frac{2}{3}}}{\sqrt{27}}\int_{-1}^{1}\sqrt{t^{3}+1}\,dt
\end{equation}
for the domain $G_z^{(2)}$.

Therefore, we have
\[
\iint_{G_z} h(x,y,1,\delta)\,dx\,dy = c_z \cdot \delta^{5/6} + O(\delta),
\]
where $c_z$ is defined by \eqref{eq_cz}.

Since
\[
\mes_2 (G_{z}\setminus M_z) = O(\delta^{1/2}), \qquad \sup_{(x,y)\in G_{z}\setminus M_z} h(x,y,1,\delta) = O(\delta^{1/2}),
\]
we obtain
\[
\iint_{G_{z}\setminus M_z}h(x,y,1,\delta)\,dx\, dy = O(\delta).
\]
Thus, Lemma \ref{lm_sgmaz} is proved.
\end{proof}

\begin{lemma}\label{lm_sgmay}
For $0<\delta\le 1/27$, we have
\begin{equation}\label{eq_sgmay}
\sigma_y(\delta) = c_y \cdot \delta^{5/6} + O(\delta),
\end{equation}
where
\begin{equation}\label{eq_cy}
c_y = 3(\sqrt[3]{3}-1)\cdot \frac{4^{2/3}}{\sqrt{27}}
\left(
2\int_{1}^{\infty} \frac{dt}{\sqrt{t^3+1}+\sqrt{t^3-1}} +
\int_{-1}^{1} \sqrt{t^3+1}\,dt
\right).
\end{equation}
\end{lemma}

\begin{proof}
Let us calculate the integral
$\iint_{G_y} h(x,1,z,\delta)\,dx\,dz$.
According to \eqref{eq_hxyzd}, for the domain $G_{y}$
we have $G_{y} = G_y^{(1)} \sqcup G_y^{(2)}$,
where
\begin{align*}
G_y^{(1)} &= \{(x,z) \in G_{y} : S(x,1,z) > 27 z^2 \delta\}, \\
G_y^{(2)} &= \{(x,z) \in G_{y} : |S(x,1,z)| \le 27 z^2 \delta\}.
\end{align*}

We consider the domain $G_y^{(1)}$. Let us apply the substitutions \eqref{er14} to the integral
\[
I_y^{(1)} := \iint_{G_y^{(1)}}h(x,1,z,\delta)dxdz =
8 \delta \int_{0}^{1} dz \int_{x_{1}(z)}^{x_{2}(z)}\frac{dx}{\sqrt{S+27 z^{2}\delta}+\sqrt{S-27 z^2\delta}},
\]
where
\[
x_1(z) = \begin{cases}
\frac{1}{3z} - \sqrt[3]{\frac{\delta}{4z}}, & -1\le z\le -z_\delta,\\
-1, & -z_\delta < z\le 1,
\end{cases}
\]
\[
x_2(z)=
\begin{cases}
1, & -1\le z < z_\delta,\\
\frac{1}{3z} - \sqrt[3]{\frac{\delta}{4z}}, & z_\delta\le z\le 1,
\end{cases}
\]
and $z_{\delta}\in [0,1]$ is the real solution of the equation $\frac{1}{3z}-\sqrt[3]{\frac{\delta}{4z}}=1$, which is equivalent to
\begin{equation}\label{er22}
 \delta=\frac{4}{z^{2}} \left( \frac{1}{3}-z \right)^{3}.
\end{equation}
Note that $z_\delta \in\left[\frac17,\frac13\right]$ for $\delta\in[0,1]$, and
$\lim_{\delta\to 0} z_{\delta} = \frac13$.

After the substitutions \eqref{er14} the limits of integration are given by
\[
 t_{1}(z)=
\begin{cases}
1, & -1\le z \le -z_{\delta},\\
\frac{1}{3}\sqrt[3]{\frac{4}{z^{2}\delta}}(1+3z), & -z_{\delta} < z \le 1;
\end{cases}
\]
\[
 t_{2}(z)=
\begin{cases}
\frac{1}{3}\sqrt[3]{\frac{4}{z^{2}\delta}}(1-3z), & -1\le z < z_{\delta},\\
1, & z_{\delta} \le z \le 1.
\end{cases}
\]

Hence the integral $I_y^{(1)}$ may be rewritten as
\[
 I_y^{(1)} = \delta^{\frac{5}{6}} \cdot \frac{2\cdot 4^{\frac{2}{3}}}{\sqrt{27}}\int_{0}^{1}\frac{dz}{z \sqrt[3]{z}} \int_{t_{2}(z)}^{t_{1}(z)}\frac{dt}{\sqrt{t^{3}+1}+\sqrt{t^{3}-1}}.
\]
The right hand side of the integral can be written as a sum of two integrals, which can be estimated as follows:
\begin{multline*}
 J_{1} := \int_{0}^{z_\delta}\frac{dz}{z \sqrt[3]{z}} \int_{t_{2}(z)}^{t_{1}(z)}\frac{dt}{\sqrt{t^{3}+1}+\sqrt{t^{3}-1}}
\le \int_{0}^{z_\delta}\frac{dz}{z \sqrt[3]{z}} \int_{t_{2}(z)}^{t_{1}(z)}\frac{dt}{\sqrt{t^{3}}}
\le\\
\le \delta^{\frac{1}{6}} \cdot 6\, \sqrt[3]{4}\, \sqrt{3}\int_{0}^{\frac{1}{3}}\frac{dz}{\sqrt{1-9z^{2}} \cdot (\sqrt{1+3z}+\sqrt{1-3z})},
\end{multline*}
\begin{multline*}
 J_{2} := \int_{z_\delta}^{1}\frac{dz}{z \sqrt[3]{z}} \int_{t_{2}(z)}^{t_{1}(z)}\frac{dt}{\sqrt{t^{3}+1}+\sqrt{t^{3}-1}} = \\
= 3(z_\delta^{-1/3}-1) \int_{1}^{\infty}\frac{dt}{\sqrt{t^{3}+1}+\sqrt{t^{3}-1}}+O(\delta^{1/6}).
\end{multline*}

It is easy to obtain from \eqref{er22} that $\left|z_\delta^{-1/3}-\sqrt[3]{3}\right| = O(\delta^{1/3})$.

Thus,
\begin{multline}\label{er17}
I_y^{(1)} = \delta^{\frac{5}{6}} \cdot \frac{2\cdot 4^{\frac{2}{3}}}{\sqrt{27}} (J_1 + J_2) = \\
= \delta^{\frac{5}{6}} \cdot \frac{2\cdot 4^{\frac{2}{3}}}{\sqrt{27}}\,
3(\sqrt[3]{3}-1) \int_{1}^{\infty}\frac{dt}{\sqrt{t^{3}+1}+\sqrt{t^{3}-1}}+O(\delta).
\end{multline}

For the domain $G_y^{(2)}$ we have
\[
I_y^{(2)} := \iint_{G_y^{(2)}} h(x,1,z,\delta)dx dz =
4 \int_{z_\delta}^1 dz \int_{x_1(z)}^{x_2(z)} \frac{\sqrt{S + 27 z^2 \delta}}{27 z^2} dx,
\]
where $x_1(z) = \frac{1}{3z} - \sqrt[3]{\frac{\delta}{4z}}$ and $x_2(z) = \min\left\{ \frac{1}{3z} + \sqrt[3]{\frac{\delta}{4z}}, \ 1 \right\}$.

Let us apply the substitution \eqref{er14} to the integral $I_y^{(2)}$.
Since $\left|z_{\pm\delta}-\frac13\right|=O(\delta^{1/3})$,
$|x_2(z)-x_1(z)| = O(\delta^{1/3})$ for $z\ge z_\delta$,
and $\sup_{(x,z)\in G_y^{(2)}} h(x,1,z,\delta) = O(\delta^{1/2})$,
we get the following asymptotics for $I_y^{(2)}$
\begin{equation}\label{er18}
 I_y^{(2)} = \delta^{\frac{5}{6}} \cdot
\frac{4^{\frac{2}{3}}}{\sqrt{27}} \cdot 3 (\sqrt[3]{3}-1) \int_{-1}^{1} \sqrt{t^3 + 1}\,dt + O(\delta^{7/6}).
\end{equation}

From \eqref{er17} and \eqref{er18}, we get
\[
\iint_{G_y} h(x,1,z,\delta)\,dx\,dz = c_y \cdot \delta^{5/6} + O(\delta),
\]
where $c_y$ is defined by \eqref{eq_cy}.

Since
\[
\mes_2 (G_y\setminus M_y) = O(1), \qquad \sup_{(x,z)\in G_y\setminus M_y} h(x,1,z,\delta) = O(\delta),
\]
we obtain
\[
\iint_{G_y\setminus M_y}h(x,1,z,\delta)\,dx\, dz = O(\delta).
\]
Hence, the proof of Lemma \ref{lm_sgmay} is completed.
\end{proof}

Using Lemmas \ref{lm_sgmaz} and \ref{lm_sgmay}, we obtain the equality
\begin{equation}\label{er19}
 \sigma(\delta) = c_1 \cdot \delta^{\frac{5}{6}} + O(\delta),
\end{equation}
where $c_1:=4(c_z + c_y)$, and $c_z$, $c_y$  are defined by \eqref{eq_cz}, \eqref{eq_cy}.

This is the desired equation, and Lemma \ref{sgm_up_bound} is proved.
\end{proof}

\section{Proof of Theorem~\ref{thm1}}

\begin{lemma}\label{prop1}
Let
\[
V_3(\delta) = \{ P \in \mathbb{R}[x] : \deg P = 3, \ H(P) \le 1, \ |D(P)| \le \delta \}.
\]
Then for $0<\delta\le 1/27$
\begin{equation}\label{er26}
\mes_{4} V_{3}(\delta) = \kappa \cdot \delta^{5/6} + O(\delta\, |\ln\delta| + \delta),
\end{equation}
where $\kappa = \frac32 c_1$, and $c_1$ is the same as in \eqref{er13};
an implicit constant in the symbol $O(\cdot)$ is absolute.
\end{lemma}
\begin{proof}
Applying Lemma \ref{lm_dmn_mes}, we have
\[
\mes_4 V_3(\delta) = \int_0^1 r^3\, \sigma\left(\frac{\delta}{r^4}\right)\, dr.
\]

Using Lemma \ref{sgm_up_bound}, we have that
the asymtotic formula
$\sigma\left(r^{-4}\delta \right) = c_1 r^{-10/3} \delta^{5/6} + O(r^{-4}\delta)$
holds on the interval $[r_0,1]$, where $r_0 := (27\delta)^{1/4}$.
Thus, we obtain
\[
\int_{r_0}^1 r^3\, \sigma\left(\frac{\delta}{r^4}\right)\, dr =
\frac32 c_1 \delta^{5/6} + O(\delta\,|\ln\delta| + \delta).
\]

On the interval $[0,r_0)$ we shall use the trivial bound $\sigma\left(r^{-4}\delta \right)\ll 1$.
Hence, we have $\int_0^{r_0} r^3\, \sigma\left(\frac{\delta}{r^4}\right)\, dr = O(\delta)$.
The lemma is proved.
\end{proof}

\begin{theorem}[\cite{Dav51_LP}]\label{thm_intp_num}
Let $\mathcal{D}\subset \mathbb{R}^d$ be a bounded region consisting of all points $(x_1,\dots,x_d)$ that satisfy all of a finite set of algebraic inequalities
\[
F_i(x_1,\dots,x_d)\ge 0, \qquad 1\le i\le k,
\]
where $F_i$ is a polynomial with real coefficients of degree $\deg F_i \le m$.
Let
\[
\Lambda(\mathcal{D}) = \mathcal{D}\cap \mathbb{Z}^d.
\]
Then
\[
\left|\#\Lambda(\mathcal{D}) - \mes_d \mathcal{D}\right| \le C \max(\bar{V}, 1),
\]
where $C$ depends only on $d$, $k$, $m$, and $\bar{V}$ is the greatest $r$--dimensional measure of any projection of $\mathcal{D}$ on a coordinate space, $1\le r\le d-1$.
\end{theorem}

Let $Q\cdot\mathcal{S}$ denote a set obtained by uniform scaling of a set $\mathcal{S}$ in $Q$ times.

It is easy to see that
\[
N(Q,X) = \#\Lambda\left(Q\cdot V_3(X/Q^4)\right).
\]
Assuming $\delta = X/Q^4$ and $\mathcal{D} = Q\cdot V_3(\delta)$ in Theorem \ref{thm_intp_num}, we obtain
\[
N(Q,X) = Q^4\cdot \mes_4 V_3(X/Q^4) + O(Q^3).
\]
Now, Theorem \ref{thm1} follows from Lemma \ref{prop1}.

By assuming $X = \gamma_3 Q^{4-2v}$ we obtain Corollary \ref{cor1}.
Here, the bounds for $v$ are a direct consequence of the condition
$Q^3 \ll Q^{2/3} X^{5/6}$,
which proves Corollary \ref{cor1}.

\section{Proof of Theorem~\ref{thm2}}

Let us define the following sets of polynomials
\begin{gather}
\mathcal{P}_{n}^*(Q)=\{P\in \mathcal{P}_{n}(Q) : P \text{ is irreducible over } \mathbb{Q} \}, \label{eq_P*Q}\\
\mathcal{P}_{n}^{**}(Q)=\{P\in \mathcal{P}_{n}(Q) : P \text{ is reducible over } \mathbb{Q} \}.
\end{gather}
Obviously, we have
$\mathcal{P}_{n}^{**}(Q) = \mathcal{P}_{n}(Q) \backslash \mathcal{P}_{n}^*(Q)$.

\begin{lemma}[{\cite[Lemma 1]{K12}}]\label{lm_P**}
The number of reducible polynomials $\#\mathcal{P}_{n}^{**}(Q)$ has the following order:
\begin{equation*}
\#\mathcal{P}_{n}^{**}(Q) \asymp_n
\begin{cases}
Q^n, & n\ge 3,\\
Q^2 \ln Q, & n = 2.
\end{cases}
\end{equation*}
\end{lemma}

Let us define the following functions
\begin{gather*}
\nu(X) = \nu(Q,X) := \#\{ P \in \mathcal{P}_3^{*}(Q) : |D(P)| = X \},\\
N^*(X) = N^*(Q,X) := \#\{ P \in \mathcal{P}_3^{*}(Q) : |D(P)| \leq X \} = \sum_{1 \le d \le X} \nu(d),\\
s(X) = s(Q,X) := \sum_{\stackrel{\scriptstyle P \in \mathcal{P}_3^{*}(Q)}{|D(P)| \le X}}|D(P)|^{-1/2}.
\end{gather*}

Firstly, we shall obtain the upper bounds for the functions $N^*(X)$ and $s(X)$, which follow directly from
Davenport's results \cite{Dav61}.
\begin{lemma}
The following upper bounds hold for the functions $N^*(Q,X)$ and $s(Q,X)$:
\begin{gather}
 N^*(Q,X) \ll Q \cdot X^{3/4}, \label{eq_N_est} \\
 s(Q,X) \ll Q \cdot X^{1/4}. \label{eq_s_est}
\end{gather}
\end{lemma}
\begin{proof}
The sum $s(X)$ may be written in form
\[
s(X) = \sum_{1\le d \le X} \nu(d)d^{-1/2}.
\]

To apply Davenport's result we need to introduce some terminology and notations.
Following \cite{Dav51_I}, two binary cubic forms with integral coefficients are said to be {\bf properly equivalent},
if one can be transformed into the other by a linear substitution with integral coefficients and determinant 1.

Let $h(D)$ denote the number of classes of properly equivalent irreducible binary cubic forms
that have discriminant $D$.

Using the formula (5) from \cite{Dav61}, we obtain
\[
\nu(d) \ll Q \cdot (h(d)+h(-d)) \cdot d^{-1/4}.
\]
This result yields the following estimates for $N^*(X)$ and $s(X)$
\begin{gather*}
N^*(X) \ll Q \sum_{1\le d \le X} (h(d)+h(-d)) \cdot d^{-1/4},\\
s(X) \ll Q \sum_{1\le d \le X} (h(d)+h(-d)) \cdot d^{-3/4}.
\end{gather*}
It was proved in \cite{Dav51_I, Dav51_II} (see formulas (3) and (1) respectively) that
\[
\sum_{1 \le |d| \le X} h(d) \ll X.
\]
Thus, by partial summation we obtain the bounds \eqref{eq_N_est} and \eqref{eq_s_est}.
Note that this lemma is a direct extension of Davenport's result (see formula (3) from~\cite{Dav61}).
\end{proof}

Now using Theorem \ref{thm1}, we shall get the asymtotic formula for the function $N^*(X)$.
\begin{lemma}\label{lm_NX_asmp}
For $X$ satisfying $c_2 Q^{14/5} \le X \le \gamma_3 Q^4$, where $c_2$ is an absolute constant,
and sufficiently large $Q$,
the function $N^*(Q, X)$ has the asymptotic formula:
\begin{equation}\label{eq_NX_asmp}
N^*(Q,X) = \kappa\, Q^{2/3} \, X^{5/6} + O\left(X\left|\ln\left(Q^4 / X\right)\right| + X + Q^3\right),
\end{equation}
where the absolute constant $\kappa$ is the same as in Theorem \ref{thm1}.
\end{lemma}
\begin{proof}
Lemma \ref{lm_P**} gives $\# \mathcal{P}_3^{**}(Q) \asymp Q^3$. Thus, with the use of Theorem \ref{thm1}, we have \eqref{eq_NX_asmp}.
\end{proof}

We will denote
\[
S(Q,X) = \sum_{\stackrel{\scriptstyle P \in \mathcal{P}_3^{*}(Q)}{|D(P)|\le X}}|D(P)|^{-1/2}.
\]
It naturally follows that
\[
S(Q,X) = \sum_{1\le d \le X}\nu(d)d^{-1/2}.
\]
The sum can be now split into two parts: $S(Q,X) = S_1(Q) + S_2(Q,X)$, where
\[
S_1(Q) = \sum_{1\le d \le c_2 Q^{14/5}}\nu(d)d^{-1/2}, \qquad
S_2(Q,X) = \sum_{c_2 Q^{14/5} < d \le X}\nu(d)d^{-1/2}.
\]

The upper bound for $S_1(Q)$ follows from \eqref{eq_s_est}
\[
S_1(Q) = s(Q, c_2 Q^{14/5}) \ll Q^{17/10}.
\]

Let us estimate $S_2(Q,X)$ by partial summation. We have
\begin{multline*}
\sum_{d = D_1}^{D_2}\nu(d)d^{-1/2} = \sum_{d = D_1}^{D_2}(N^*(d)-N^*(d-1))d^{-1/2} =\\
= \frac{N^*(D_2)}{\sqrt{D_2+1}}-\frac{N^*(D_1-1)}{\sqrt{D_1}}+\sum_{d = D_1}^{D_2}N^*(d)(d^{-1/2}-(d+1)^{-1/2}),
\end{multline*}
For any $d\ge 2$ we have
\[
(d+1)^{-1/2} = d^{-1/2} - \frac12 d^{-3/2} + O(d^{-5/2}).
\]
Assumming $D_1 = c_2 Q^{14/5}$ and $D_2 = X$, we obtain:
\begin{gather*}
\frac{N^*(X)}{\sqrt{X+1}} = \kappa\, Q^{2/3} \, X^{1/3} + O\left(X^{1/2}\left|\ln\left(Q^4 / X\right)\right| + X^{1/2} + Q^3 X^{-1/2}\right),\\
\frac{N^*(D_1-1)}{\sqrt{D_1}} \ll Q^{3-7/5} \ll Q^{8/5}.
\end{gather*}
Using Lemma \ref{lm_NX_asmp}, we obtain
\begin{multline*}
\sigma = \sum_{d = D_1}^{D_2}N^*(d)(d^{-1/2}-(d+1)^{-1/2}) = 
\frac{\kappa}{2} Q^{2/3} \cdot \sum_{d = D_1}^{D_2}d^{-2/3} \ + \\
+ O\left(\sum_{d=D_1}^{D_2}\left(d^{-1/2}\left|\ln\left(Q^4 / d\right)\right| + d^{-1/2} + Q^3 d^{-3/2} + Q^{2/3} d^{-5/3} \right)\right).
\end{multline*}
Since
\[
\sum_{d = D_1}^{X}d^{-2/3} = \frac13 X^{1/3} + O(D_1^{1/3} + X^{-2/3}),
\]
we have
\[
\sigma = \frac{\kappa}{6} Q^{2/3} X^{1/3} + O\left(X^{1/2}\left|\ln\left(Q^4 / X\right)\right| + X^{1/2} + Q^{8/5}\right).
\]
Theorem \ref{thm2} is proved.

\section{Some applications}

In this section, we give an illustration how some results on Hausdorff's dimension and Mahler's
Problem for cubic polynomials could be obtained from our estimates. These results are
well-known and have been solved \cite{BS, Ber83, Vol}, but in due time, it was hard problems. We show here
how these problems could be solved in cubic case by simple using of our result.

\noindent\textbf{Mahler's Conjecture \cite{S67}.} \textit{
Let $\mathcal{L}_{n}(w)$ be the set of real numbers such that the inequality
\begin{equation}\label{ap1}
 |P(x)|<H(P)^{-w},~w>n,
\end{equation}
has infinitely many solutions in integral polynomials with $\deg P\le n$. Then
\[
\mes_1 \mathcal{L}_{n}(w)=0.
\]
}

Let $x\in\mathcal{L}_3(w)$, and $P\in\mathbb{Z}[x]$, $\deg P\le n$, be a solution of \eqref{ap1}.
If $\alpha_{1}$ is the root of $P(x)$ closest to $x$, then it is known (see \cite{S67}) that
\begin{equation}\label{ap2}
 |x-\alpha_{1}|<6H(P)^{-w}|D(P)|^{-\frac{1}{2}}.
\end{equation}

Let $\mathcal{L}(t,w)$ be the set of real numbers $x$ such that the inequality \eqref{ap1} has solutions in polynomials belonging to the class
\[
 \mathcal{P}_{t}=\{ P \in \mathcal{P}_{3}^{*}(2^t) \ : \ 2^{t-1} < H(P) \},
\]
where the set $\mathcal{P}_{3}^{*}(Q)$ is defined according to \eqref{eq_P*Q}.

Let us cover the set $\mathcal{L}(t,w)$ by intervals
$I_{\alpha} = \{x\in\mathbb{R} : |x-\alpha|<6H(P)^{-w}|D(P)|^{-\frac{1}{2}}\}$,
where $P\in \mathcal{P}_t$, and $\alpha$ is a real root of $P$.
Let us consider the series
\begin{multline*}
 S_{1}:=\sum_{t=1}^{\infty}
\sum_{\substack{
P\in\mathcal{P}_{t}\\
 c_2 2^{2.8t}<|D(P)|
}}
\!\!\left(2^{-(t-1)w}|D(P)|^{-\frac{1}{2}}\right)^{\frac{4}{w+1}} \ll \\
\ll \,
\sum_{t=1}^{\infty} 2^{-\frac{4w}{w+1}t} \!\sum_{c_2 2^{2.8t} < D \le \gamma_3 2^{4t}} N^*(D) D^{-1-\frac{2}{w+1}}.
\end{multline*}

By Lemma \ref{lm_NX_asmp}, we have $N^*(2^t, D) \ll 2^{2t/3} D^{5/6}$. This implies
\[
 S_{1} \ll \sum_{t=1}^{\infty} 2^{t\left(-\frac{4w}{w+1}+4-\frac{8}{w+1}\right)} =
\sum_{t=1}^{\infty} 2^{-\frac{4}{w+1} t} < \infty.
\]

Considering the case  when $|D(P)| \le c_2 2^{2.8t}$, we use the trivial upper bound $N^*(2^t, D) \ll 2^{3t}$. Hence, the series
\begin{multline*}
 S_{2}:=
\sum_{t=1}^{\infty}
\sum_{\substack{
P\in\mathcal{P}_{t}\\
 |D(P)|\le c_2 2^{2.8t}
}
}
\!\!\left(2^{-(t-1)w}|D(P)|^{-\frac{1}{2}}\right)^{\frac{4}{w+1}} \ll \\
\ll \,
\sum_{t=1}^{\infty} 2^{-\frac{4w}{w+1}t} \!\!\sum_{1 \le D \le c_2 2^{2.8t}} \!2^{3t} D^{-1-\frac{2}{w+1}} \ll
\sum_{t=1}^{\infty} 2^{\frac{3-w}{w+1}t}
\end{multline*}
converges for $w > n$.

Thus, for the Hausdorff dimension we have $\dim_{H} \mathcal{L}_{3}(w)\le \frac{4}{w+1}$ (that agrees with the fact that $\dim_{H} \mathcal{L}_{n}(w)= \frac{n+1}{w+1}$, see \cite{BS,Ber83}).
This leads to  $\mes_1 \mathcal{L}_{3}(w)=0$.

\subsection*{Acknowledgements}
The first and third authors would like to thank the University of Bielefeld, where a substantial part of this work was done, for providing a stimulating research environment during their visits supported by SFB 701.

\bigskip

{\small Dzianis Kaliada}\\
{\footnotesize
{Institute of Mathematics, National Academy of Sciences of Belarus,\\
220072 Minsk, Belarus}\\
E-mail: koledad@rambler.ru
}

\bigskip

{\small Friedrich G\"otze}\\
{\footnotesize
{University of Bielefeld,\\
33501, Bielefeld, Germany}\\
E-mail: goetze@math.uni-bielefeld.de
}

\bigskip

{\small Olga Kukso}\\
{\footnotesize
{Institute of Mathematics, National Academy of Sciences of Belarus,\\
220072 Minsk, Belarus}\\
E-mail: olga\_kukso@tut.by
}

\end{document}